\documentclass{amsart}
\usepackage[hidelinks]{hyperref}
\usepackage[inline]{enumitem}
\usepackage{amsthm, amsmath, amssymb, colonequals, comment}

\newtheorem{corollary}{Corollary}
\newtheorem{lemma}{Lemma}
\newtheorem{proposition}{Proposition}
\newtheorem{theorem}{Theorem}

\theoremstyle{definition}
\newtheorem{definition}{Definition}
\newtheorem{example}{Example}

\theoremstyle{remark}


\begin{document}

\title{Ideals in hom-associative Weyl algebras}
\author{Per B\"ack}
\address[Per B\"ack]{Division of Mathematics and Physics, M\"alar\-dalen  University,  Box  883,  SE-721  23  V\"aster\r{a}s, Sweden}
\email[corresponding author]{per.back@mdu.se}

\author{Johan Richter}
\address[Johan Richter]{Department of Mathematics and Natural Sciences, Blekinge Institute of Technology, SE-371 79 Karlskrona, Sweden}
\email{johan.richter@bth.se}


\begin{abstract}
We introduce hom-associative versions of the higher order Weyl algebras, generalizing the construction of the first hom-associative Weyl algebras. We then show that the higher order hom-associative Weyl algebras are simple, and that all their one-sided ideals are principal. 
\end{abstract}

\maketitle

\section{Introduction}
Dixmier~\cite{Dix70} has shown that every left (right) ideal of the first Weyl algebra $A_1$ over a field $K$ of characteristic zero can be generated by two elements. Later, and more generally, Stafford~\cite{Sta76} has shown that every left (right) ideal of a simple left (right) Noetherian ring with Krull dimension $n$ can be generated by $n+1$ elements; in particular, this result applies to the $n$th Weyl algebra $A_n$ over $K$. Stafford~\cite{Sta78} has further improved this result for $A_n$ and shown that every left (right) ideal of $A_n$ over $K$ can be generated by two elements, a classical result today more commonly known as \emph{Stafford’s theorem}. 

In this article, we introduce \emph{higher order hom-associative Weyl algebras} as hom-associative deformations of the higher order Weyl algebras over $K$ and consider what a hom-associative version of Stafford's theorem would look like. We prove that, subject to a non-triviality condition on the deformation, the higher order hom-associative Weyl algebras are simple (\autoref{cor_simple}) and that all their one-sided ideals are principal (\autoref{thm:principal}). 

\section{Preliminaries}\label{sec:prel}
Throughout this article, we denote by $\mathbb{N}$ the set of non-negative integers. By a \emph{non-as\-so\-cia\-tive algebra} over an associative, commutative, and unital ring $R$, we mean an $R$-algebra $A$ which is not necessarily associative and not necessarily unital. 

\subsection{Hom-associative algebras}\label{subsec:hom}\emph{Hom-associative algebras} were introduced in \cite{MS08} as non-associative algebras with a ``twisted'' associativity condition. In particular, by using the commutator as a bracket, any hom-associative algebra gives rise to a \emph{hom-Lie algebra}; the latter introduced in \cite{HLS03} as a  generalization of a Lie algebra, now with a twisted Jacobi identity.

\begin{definition}[Hom-associative algebra]
 A \emph{hom-associative algebra} over an associative, commutative, and unital ring $R$, is a non-associative $R$-algebra $A$ with an $R$-linear map $\alpha$, where for all $a,b,c\in A$, the \emph{hom-associative condition} holds,
 \begin{equation*}
    \alpha(a)(bc)=(ab)\alpha(c).
 \end{equation*}
 \end{definition}
 Since $\alpha$ in the above definition ``twists'' the associativity condition, it is referred to as a \emph{twisting map}.

For hom-associative algebras it is usually too restrictive to expect them to be unital. Instead, a related condition, called \emph{weak unitality}, is of interest. 

\begin{definition}[Weak unitality]\label{def:weak-hom}
Let $A$ be a hom-associative algebra. If for all $a\in A$, $ea=ae=\alpha(a)$ for some $e\in A$, we say that $A$ is \emph{weakly unital} with \emph{weak identity element} $e$. 
\end{definition}

The so-called \emph{Yau twist} gives a way of constructing (weakly unital) hom-as\-so\-cia\-tive algebras from (unital) associative algebras.
 
\begin{proposition}[\cite{FG09, Yau09}]\label{prop:star-alpha-mult} Let $A$ be an associative algebra and let $\alpha$ be an algebra endomorphism on $A$. Define a new product $*$ on $A$ by $a*b\colonequals\alpha(ab)$ for any $a,b\in A$. Then $A$ with product $*$, called the \emph{Yau twist} of $A$, is a hom-associative algebra with twisting map $\alpha$. If $A$ is unital with identity element $1_A$, then the Yau twist of $A$ is weakly unital with weak identity element $1_A$.
\end{proposition}

By a left (right) \emph{hom-ideal} in a hom-associative algebra, we mean a left (right) ideal that is also invariant under the twisting map. If the algebra is weakly unital, then all ideals, one-sided and two-sided, are automatically invariant under the twisting map. 

\subsection{The $n$th  Weyl algebra}
The \emph{$n$th Weyl algebra}, $A_n$, over a field $K$ of characteristic zero is the free, associative, and unital algebra with generators $x_1, x_2, \ldots , x_n$ and $y_1, y_2, \ldots , y_n$, $K\langle x_1,x_2, \ldots, x_n, y_1, y_2,\ldots, y_n\rangle$, modulo the commutation relations 
\begin{align*}
    x_ix_j&=x_jx_i \text{ for all } i,j \in \{1,2, \ldots, n\},\\
    y_iy_j&=y_jy_i \text{ for all } i,j \in \{1,2, \ldots, n\},\\
    x_iy_j&=y_jx_i \text{ for all } i,j \in \{1,2, \ldots, n\}\text{ such that }i\neq j,\\
    x_iy_i&=y_ix_i+1\text{ for all } i \in \{1,2, \ldots, n\}. 
\end{align*}

\subsection{The first hom-associative Weyl algebras}
In \cite{BRS18}, a family of \emph{hom-as\-so\-cia\-tive Weyl algebras} $\{A_1^k\}_{k\in K}$ was constructed as a generalization of $A_1$ to the hom-associative setting (see also \cite{BR22} for the case when $K$ has prime characteristic), including $A_1$ as the member corresponding to $k=0$. The definition of $A_1^k$ is as follows:
\begin{definition}[The first hom-associative Weyl algebra]
Let $\alpha_k$ be the $K$-auto\-mor\-phism on $A_1$ defined by $\alpha_k(x)\colonequals x$, $\alpha_k(y)\colonequals y+k$, and $\alpha_k(1_{A_1})\colonequals 1_{A_1}$ for any $k\in K$. The \emph{first hom-associative Weyl algebra} $A_1^k$ is the Yau twist of $A_1$ by $\alpha_k$. 
\end{definition}
For each $k\in K$, we thus get a hom-associative Weyl algebra $A_1^k$ which is weakly unital with weak identity element $1_{A_1}$. In \cite{BRS18}, it was proven that $A_1^k$ is simple for all $k\in K$. In \cite{BR20}, the study of $A_1^k$ was continued. The morphisms and derivations on $A_1^k$ were characterized, and an analogue of the famous \emph{Dixmier conjecture}, first introduced by Dixmer \cite{Dix68}, was proven. It was also shown that $A_1^k$ is a formal deformation of $A_1$ with $k$ as deformation parameter, this in contrast to the associative setting where $A_1$ is formally rigid and thus cannot be formally deformed. 

\subsection{Monomial orderings}
We introduce an ordering, the so-called \emph{graded lexicographic ordering} on $\mathbb{N}^n$, where a vector is larger than another vector if it has larger sum of all its elements. In case of a tie, we apply \emph{lexicographic ordering}, that is, $(1,0,0\ldots,0)>(0,1,0,\ldots,0)>\cdots >(0,0,0,\ldots,1)$. For example, $(0,0,3)>(1,1,0)>(0,2,0)>(0,1,0)>(0,0,0)$. Note that this is a total ordering on $\mathbb{N}^n$ and that any subset has a smallest element. Furthermore, it is impossible to find an infinite decreasing sequence in $\mathbb{N}^n$. Note that this gives an ordering of the monomials in $K[y_1,y_2,\ldots,y_n]$.

Any $p\in A_n$ can be written as $\sum_{l\in\mathbb{N}^n}p_lx_1^{l_1}x_2^{l_2}\cdots x_n^{l_n}$ where $p_l\in K[y_1,y_2,\ldots,y_n]$,  $l=(l_1,l_2,\ldots,l_n)\in\mathbb{N}^n$, and only finitely many of the $p_l$ are non-zero. We define $\deg_x(p)$ as the largest $l$ in graded lexicographic order such that $p_l$ is non-zero, and $L(p)=p_{\deg_x(p)}$. We also define $\deg_y$ in a similar way. We will often write $\deg_y(p)=0$, where $0$ should be understood as the zero vector of appropriate dimension. 

\section{Ideals in higher order hom-associative Weyl algebras}
We define the \emph{$n$th hom-associative Weyl algebra} in anaology with how the first hom-associative Weyl algebra is defined. 
\begin{definition}[The $n$th hom-associative Weyl algebra] Let $K$ be a field of characteristic zero and let $k=(k_1,k_2,\ldots,k_n)\in K^n$. Define the $K$-automorphism $\alpha_k$ on $A_n$ by $\alpha_k(x_i)\colonequals x_i$, $\alpha_k(y_i)\colonequals y_i+k_i$, and $\alpha_k(1_{A_n})\colonequals 1_{A_n}$ for $1\leq i\leq n$. The \emph{$n$th hom-associative Weyl algebra} $A_n^k$ is the Yau twist of $A_n$ by $\alpha_k$.
\end{definition}

\begin{proposition}\label{prop_hom-ideals-from-higher-Weylalgebra}
If $I$ is a left (right) ideal of $A_n$, then $I$ is a left (right) ideal of $A_n^k$ if and only if $\alpha_k(I)\subseteq I$.	
\end{proposition}

\begin{proof}
We show the left case; the right case is similar. To this end, let $I$ be a left ideal of $A_n$. If $\alpha_k(I)\subseteq I$, $p\in A_n$, and $q\in I$, then $p*q=\alpha_k(pq)\in\alpha_k(I)\subseteq I$.  

If $I$ is a left ideal of $A_n^k$ and $q\in I$, then $\alpha_k(q)=1_{A_n}*q\in I$, so $\alpha_k(I)\subseteq I$.
\end{proof}

\begin{example}
Let $I$ be the left ideal of $A_1$ generated by $x^n$ for some $n\in\mathbb{N}_{>0}$. Then $I$ is a non-trivial left ideal (for example, $y\not\in I$). Any element in $I$ may be written as $px^n$ for some $p\in A_1$. We have $\alpha_k(px^n)=\alpha_k(p)\alpha_k(x^n)=\alpha_k(p)x^n\in I$, so $\alpha_k(I)\subseteq I$. Similarly, if $I$ is the right ideal of $A_1$ generated by $x^n$, then $I$ is a non-trivial right ideal such that $\alpha_k(I)\subseteq I$. By \autoref{prop_hom-ideals-from-higher-Weylalgebra}, $I$ is a non-trivial left (right) ideal of $A_1^k$.  
\end{example}

By the next example, not all left (right) ideals of $A_1$ are left (right) ideals of $A_1^k$ when $k\neq0$.

\begin{example}
Let $I$ be the left (right) ideal of $A_1$ generated by $y$. Then $x\not\in I$, so $I\neq A_1$. Assume that $k\neq0$ and $\alpha_k(I)\subseteq I$. Then $y+k=\alpha_k(y)\in I$, so $k=(y+k)-y \in I$. Hence $1_{A_1}\in I$, which implies $I=A_1$; a contradiction. By \autoref{prop_hom-ideals-from-higher-Weylalgebra}, $I$ is not a left (right) ideal of $A_1^k$.
\end{example}

\begin{lemma}\label{lem_higheridealInvariants}
If $I$ is a left (right) ideal of $A_n^k$ where $k_1k_2\cdots k_n\neq0$, then $\alpha_k(I)=I$.
\end{lemma}

\begin{proof}
Since any left (right) ideal $I$ of $A_n^k$ is also a left (right) hom-ideal,  $\alpha_k(I) \subseteq I$.

Now, let $I$ be a left ideal of $A_n^k$. If $0 \neq p \in I$, we claim that we can find an element $p' \in I$ such that $\deg_x(p')=\deg_x(p)$ and $L(p')=1$. If $L(p)=c\in K$, we can take $p'= c^{-1}*p$. Otherwise, we note that $\deg_x(\alpha_k(p)-p) = \deg_x(p)$, and that $L(\alpha_k(p))$ and $L(p)$ have the same leading term using our monomial ordering on $K[y_1,y_2,\ldots,y_n]$. Thus, $L(\alpha_k(p)-p)$ has lower degree than $L(p)$ using our monomial ordering. We can repeat this process until we get an element in $I$ with a constant as leading coefficient w.r.t. $x_1,x_2,\ldots, x_n$, and with the same degree in $x_1,x_2,\ldots,x_n$ as $p$. 

Now suppose $I\not\subseteq \alpha_k(I)$. Then there is at least one element in $I$ that does not belong to $\alpha_k(I)$. Pick such an element, $q$, of lowest possible degree w.r.t. $x_1,x_2,\ldots,x_n$. Find an element $q'\in I$ such that $\deg_x(q')=\deg_x(q)$ and $L(q')=1$. Set $r\colonequals\alpha_k^2(\alpha_k^{-2}(L(q))q')= \alpha_k(\alpha_k^{-2}(L(q))*q')$. Note that $\deg_x(r)= \deg_x(q)$ and that $L(r)=L(q)$. Since $\alpha_k^{-2}(L(q))*q' \in I$, we have $r\in \alpha_k(I)\subseteq I$. Hence $q-r\in I$, and by the minimality of $q$, we must have $q-r \in \alpha_k(I)$. However, this would imply that also $q \in \alpha_k(I)$, which is a contradiction. 

Now let $I$ be a right ideal of $A_n^k$. If $0\neq p \in I$, we can find an element $p' \in I$ such that $\deg_x(p')=\deg_x(p)$ and $L(p')=1$. If $L(p)=c\in K$, we can take $p'=p* c^{-1}$. Otherwise we proceed like in the left case.  

Now suppose $I\not\subseteq \alpha_k(I)$. Then there is at least one element in $I$ that does not belong to $\alpha_k(I)$. Pick such an element, $q$, of lowest possible degree w.r.t. $x_1,x_2,\ldots,x_n$. Find an element $q'\in I$ such that $\deg_x(q')=\deg_x(q)$ and $L(q')=1$. Set $r\colonequals\alpha_k^2(q'\alpha_k^{-2}(L(q)))= \alpha_k(q'*\alpha_k^{-2}(L(q)))$. As in the left case, we get $r\in \alpha_k(I)$ and $q-r \in I$. By the minimality of $q$ we get $q-r \in \alpha_k(I)$, which gives the contradiction $q\in \alpha_k(I).$
\end{proof}

\begin{proposition}\label{prop_idealsOfA_n^k}
Any left (right) ideal of $A_n^k$ for $k_1k_2\cdots k_n\neq0$ is a left (right) ideal of $A_n$.
\end{proposition}

\begin{proof}
Let us prove the left case; the right case is similar. To this end, suppose $I$ is a left ideal of $A_n^k$. To show that $I$ is also a left ideal of $A_n$, it is enough to show that $A_nI \subseteq I$. Since $I$ is a left ideal of $A_n^k$, we know that $A_n*I\subseteq I$. Moreover, $A_n*I=\alpha_k(A_nI)$, so $A_n I \subseteq \alpha_k^{-1}(I)$. By \autoref{lem_higheridealInvariants}, $\alpha_k(I) = I$, so $\alpha_k^{-1}(I)=I$.
\end{proof}

\begin{corollary}\label{cor_simple}
$A_n^k$ is simple for $k_1k_2\cdots k_n\neq0$.    
\end{corollary}

\begin{proof}
Let $I$ be an ideal of $A_n^k$ for $k_1k_2\cdots k_n\neq0$. By \autoref{prop_idealsOfA_n^k}, $I$ is also an ideal of $A_n$. Since it is well known that $A_n$ is simple, $I$ must be trivial.        
\end{proof}

\begin{corollary}\label{cor_homStafford}
Any left (right) ideal of $A_n^k$ for $k_1k_2\cdots k_n\neq0$ is generated by two elements.
\end{corollary}

\begin{proof}
Let $I$ be a left ideal of $A_n^k$. Since $I$ is also a left ideal of $A_n$, we know that it is generated as an ideal of $A_n$ by two elements, say $p$ and $q$. We want to show that $\alpha^{-1}(p)$ and $\alpha^{-1}(q)$ generate $I$ as a left ideal of $A_n^k$. If $r\in I$, then there are $a,b \in A_n$ such that
$ r= ap+bq = \alpha(\alpha^{-1}(a)\alpha^{-1}(p)+\alpha^{-1}(b)\alpha^{-1}(q)) = \alpha^{-1}(a)*\alpha^{-1}(p)+\alpha^{-1}(b)*\alpha^{-1}(q)$. Clearly this shows that $\alpha^{-1}(p)$ and $\alpha^{-1}(q)$ generate $I$ as a left ideal of $A_n^k$. (That they are elements of $I$ follows from \autoref{lem_higheridealInvariants}.)

The right case is similar.  
\end{proof}

\begin{lemma}\label{lem_idealIsPrincipal}
Any left (right) ideal, $I$, of $A_n$ generated by elements $p_1, p_2, \ldots, p_m$ with $\deg_y(p_1)=\deg_y(p_2)=\cdots =\deg_y(p_m)=0$ is a principal left (right) ideal of $A_n^k$ if $k\neq0$.
\end{lemma}

\begin{proof} Assume, without loss of generality, that $k_1\neq 0$. We first show that $I$ is a left ideal of $A_n^k$. If $q\in I$, then we can write $ q= \sum_{i=1}^{m} r_i p_i $ for some $r_i \in A_n$ and $\alpha_k(q) =\sum_{i=1}^{m} \alpha_k(r_i) p_i \in I $. Hence, by \autoref{prop_hom-ideals-from-higher-Weylalgebra}, $I$ is a left ideal of $A_n^k$. It remains to show it is a principal left ideal of $A_n^k$. 

We will proceed by induction, and in fact we will show that  $I$ of $A_n$ is generated as a left ideal of $A_n^k$ by $t=p_1+y_1p_2+\cdots +y_1^{m-1}p_{m}$. To handle the case $m=2$ set $t\colonequals p_1+y_1p_2$. Let $J$ be the left ideal of $A_n^k$ generated by $t$. Then $\alpha_k(t)=p_1+y_1p_2+k_1p_2\in J$, so $p_2=k_1^{-1}*(\alpha_k(t)-t)\in J$, and hence $p_1=t-(y_1-k_1)*p_2\in J$. Hence $I=J$. 

Now assume we have proven the result for $m$ and wish to prove it for $m+1$. Set 
$ t\colonequals p_1+y_1p_2+\cdots +y_1^{m}p_{m+1}$. Let $J$ be the left ideal of $A_n^k$ generated by $t$. Obviously $J\subseteq I$. 
Note that 
$t_1\colonequals\alpha_k(t)-t = r_{1,2}p_2+r_{1,3}p_3+ \cdots +r_{1,m+1}p_{m+1}$,
where $r_{1,i} \in K[y_1]$ and $\deg_{y_1}(r_{1,i})= i-2$ for all $i\in \{ 2,3, \ldots, m+1\}.$ Also note that $t_1 \in J$. We can then set $t_2\colonequals \alpha_k(t_1)-t_1$ and note that
$ t_2 =r_{2,3}p_3+ \cdots +r_{2,m+1}p_{m+1}$
where $r_{2,i} \in K[y_1]$ and $\deg_{y_1}(r_{2,i})= i-3$ for all $i\in \{3,4, \ldots, m+1\}.$ Also $t_2\in J$. Proceeding in a similar way, we get an element $t_m \in J$ such that $t_m = r_{m,m+1}p_{m+1}$, where $r_{m,m+1}\in K$ and $r_{m,m+1}\neq 0$. Thus $p_{m+1} \in J$ and $p_1+y_1p_2+\cdots +y_1^{m-1}p_{m} \in J$, so by the induction assumption, $J=I$. 

The right case is similar; one sets $t\colonequals p_1+p_2y_1+\cdots p_my_1^{m-1} $ instead. 
\end{proof}

\begin{lemma}\label{lem_ConstantGenerators}
For any $p \in A_n$, there are $q_1, q_2, \ldots, q_{m} \in A_n$ with $\deg_y(q_1)=\deg_y(q_2)$ $=\cdots =\deg_y(q_m)=0$, such that the left (right) ideal of $A_n^k$ for $k_1k_2\cdots k_n\neq0$ generated by $p$ equals the left (right) ideal of $A_n^k$ generated by $q_1, q_2,\ldots, q_m$.   
\end{lemma}

\begin{proof}
  Let $I$ be the left ideal of $A_n^k$ generated by $p$. Set $p=\sum_{a \in \mathbb{N}^n} p_a x_1^{a_1}x_2^{a_2}\ldots x_n^{a_n}$, where each $p_a \in K[y_1, y_2,\ldots, y_n]$, and let $E(p) = \{ a \in \mathbb{N}^n \, | \, p_a \neq 0 \}$. We prove the lemma by induction over $|E(p)|$. 

  We begin with the case when $|E(p)|=1$. If $\deg_y(p)=0$, we are done. Otherwise, set $p' = p-\alpha_k(p).$ Then $E(p')= E(p)$ and $\deg_y(p')<\deg_y(p)$. Repeat as necessary until we find $q_1 \in I$ such that $E(q_1)=E(p)$ and $\deg_y(q_1)=0$. Then there is $r \in K[y_1, y_2,\ldots , y_n]$ such that $p=rq_1, $ so $p$ is in the left ideal of $A_n$ generated by $p$, and thus also in the left ideal of $A_n^k$ generated by $q_1$. Hence the left ideal of $A_n^k$ generated by $q_1$ is $I$. 

  Now suppose we have proven the lemma when $|E(p)|\leq \ell$. Assume $|E(p)|=\ell +1$. If $\deg_y(p)=0$, we are done, so let $\deg_y(p)>0$.  Set $p' = p-\alpha_k(p).$ Note that $E(p') \subseteq E(p),$ $p'\neq 0$, and that if $E(p')=E(p)$, then $\deg_y(p')<\deg_y(p).$ Repeat as necessary until we find $q_1 \in I$ such that $E(q_1) \subseteq E(p), q_1 \neq 0$, and $\deg_y(q_1)= 0$. We can then find $r\in K[y_1, y_2,\ldots , y_n]$ such that $E(p-rq_1) \subsetneq E(p)$. By the induction assumption there are $q_2, q_3,\ldots q_m$ with $\deg_y(q_2)=\deg_y(q_3)=\cdots = \deg_y(q_m)=0$ that generate the same left ideal of $A_n^k$ as $p-rq_1$. Then $q_1, q_2, \ldots, q_m$ are elements that generate $I$ as a left ideal of $A_n^k$ and $\deg_y(q_1)=\deg_y(q_2)=\cdots=\deg_y(q_m)=0$. 

The right case is similar. 
\end{proof}
    
\begin{theorem}\label{thm:principal}
Any left (right) ideal of $A_n^k$ for $k_1k_2\cdots k_n\neq0$ is principal.  
\end{theorem}

\begin{proof}
Let $I$ be a left ideal of $A_n^k$. We know it is generated by elements $p,q$ as a left ideal of $A_n$. By \autoref{lem_ConstantGenerators}, we can find $p_1, p_2, \ldots, p_{\ell}$ and $q_1, q_2, \ldots, q_m$ such that the $p_i$ generate the same left ideal of $A_n^k$ as $p$, the $q_i$ generate the same left ideal of $A_n^k$ as $q$, and $\deg_y(p_1)=\deg_y(p_2)=\cdots =\deg_y(p_{\ell})=\deg_y(q_1)=\deg_y(q_2)= \cdots = \deg_y(q_m) =0$. Clearly $p_1,p_2,\ldots , p_{\ell}, q_1,q_2, \ldots, q_m$ generate $I$ as a left ideal of $A_n^k$. By \autoref{lem_idealIsPrincipal} we are done.

The right case is similar.     
\end{proof}

\end{document}